\documentclass[12pt,reqno]{amsart}
\usepackage[margin=1in]{geometry}       
\usepackage{tikz}
\usetikzlibrary {arrows.meta}

\usepackage{graphicx}
\usepackage{amssymb}
\usepackage{aliascnt}
\usepackage{doi}
\usepackage{epstopdf}
\usepackage{pgfplots}
\usepgflibrary {shadings} 

\newcommand\myarrowL[2]{
\draw[arrows = {-Stealth[length=7pt, inset=4pt, width=7pt]},  line width=0.7pt] #1 to[bend left] #2;}

\newcommand{\arxiv}[1]{%
\href{https://arxiv.org/abs/#1}{ArXiv:#1}}

\usepackage[super]{nth}

\usepackage{psfrag}

\newtheorem{theorem}{Theorem}[section]

\newaliascnt{lemma}{theorem}
\newtheorem{lemma}[lemma]{Lemma}
\aliascntresetthe{lemma}

\newaliascnt{proposition}{theorem}
\newtheorem{proposition}[proposition]{Proposition}
\aliascntresetthe{proposition}

\newaliascnt{corollary}{theorem}

\aliascntresetthe{corollary}

\newaliascnt{conjecture}{theorem}

\aliascntresetthe{conjecture}

\newaliascnt{openQ}{theorem}

\aliascntresetthe{openQ}

\newaliascnt{quest}{theorem}

\aliascntresetthe{quest}

\newaliascnt{questx}{conjx}

\aliascntresetthe{questx}

\theoremstyle{definition}

\newaliascnt{defn}{theorem}

\aliascntresetthe{defn}

\newaliascnt{example}{theorem}

\aliascntresetthe{example}

\newaliascnt{rem}{theorem}

\aliascntresetthe{rem}

\makeatletter
\def\tagform@#1{\maketag@@@{\ignorespaces#1\unskip\@@italiccorr}}
\let\orgtheequation\theequation
\def\theequation{(\orgtheequation)}
\makeatother
\def\equationautorefname~{}

\newcommand{\B}{{\mathbb B}}
\newcommand{\C}{{\mathbb C}}
\newcommand{\D}{\mathbb{D}}
\newcommand{\e}{\varepsilon}
\newcommand{\R}{{\mathbb R}}
\newcommand{\Sp}{{\mathbb S}}
\DeclareMathOperator{\dist}{dist}

\begin{document}
	\title[Two disks maximize third Robin eigenvalue]{Two disks maximize the third Robin eigenvalue: positive parameters}


	\begin{abstract}
		The third eigenvalue of the Robin Laplacian on a simply-connected planar domain of given area is bounded above by the corresponding eigenvalue of a disjoint union of two equal disks, for Robin parameters in $[-4\pi,4\pi]$. This sharp inequality was known previously only for negative parameters in $[-4\pi,0]$, by Girouard and Laugesen. Their proof fails for positive Robin parameters because the second eigenfunction on a disk has non-monotonic radial part. This difficulty is overcome for parameters in $(0,4\pi]$ by means of a degree-theoretic approach suggested by Karpukhin and Stern that yields suitably orthogonal trial functions.   
	\end{abstract}
	
\author[]{Hanna N. Kim and Richard S. Laugesen}
\address{\hspace*{-1cm} Department of Mathematics, University of North Carolina, Chapel Hill, NC 27599, U.S.A.}
\email{hannakim@unc.edu}
\address{\hspace*{-1cm} Department of Mathematics, University of Illinois, Urbana,
	IL 61801, U.S.A.}
\email{Laugesen@illinois.edu}

	\maketitle
	

\section{\bf Introduction} \label{sec:intro}

The Robin eigenvalue problem for a bounded Lipschitz planar domain $\Omega \subset \R^2$ consists of finding eigenvalues $\lambda=\lambda(\Omega;\alpha/L)$ of the Laplacian for which an eigenfunction $u \not \equiv 0$ exists satisfying
\[
\begin{split}
\Delta u + \lambda u & = 0 \quad \text{in $\Omega$,} \\
\partial_\nu u + \frac{\alpha}{L} u & = 0 \quad \text{on $\partial \Omega$,} 
\end{split}
\]
where $\partial_\nu u$ is the outward normal derivative, $L$ is the length (perimeter) of the boundary $\partial \Omega$, and $\alpha \in \R$ is the Robin parameter. We choose in the boundary condition to divide $\alpha$ by $L$ in order to match the inverse length scale of the normal derivative $\partial_\nu u$. Consequently, the $\alpha$ parameter below in \autoref{th:main} belongs to a universal interval that is independent of the shape of $\Omega$ and of its area $A$. 

Physically, Robin eigenvalues represent rates of decay to equilibrium for heat in a partially insulated region, or frequencies of vibration for a membrane under elastically restoring boundary conditions. 

Sharp upper bounds are known for the first three Robin eigenvalues. After multiplying the eigenvalue by the area to obtain a scale invariant quantity, $\lambda_1(\Omega; \alpha/L)A$ is known to be maximal for a degenerate rectangle, for each $\alpha \in \R$; see \cite[Theorem A]{FreiLauS2020}. The second eigenvalue $\lambda_2(\Omega; \alpha/L)A$ is maximal among simply-connected domains for the disk whenever $\alpha \in [-2\pi,2\pi]$, as shown by Freitas and Laugesen \cite[Theorem B]{FreiLauS2020}. While this interval of $\alpha$-values could perhaps be enlarged, the disk certainly cannot be the maximizer for all $\alpha$, because the Robin spectrum converges to the Dirichlet spectrum as $\alpha \to \infty$ and the lowest Dirichlet eigenvalue can be made arbitrarily large by considering thin domains of fixed area.  

For the third Robin eigenvalue, Girouard and Laugesen \cite[Theorem 1.1]{GirLau2021} proved a sharp upper bound when $\alpha$ lies in the interval $[-4\pi,0]$: they proved $\lambda_3(\Omega;\alpha/L) A$ is maximized by a disjoint union of two equal disks, in the sense that this maximum value is approached in the limit as a simply-connected domain degenerates to a union of disks. 

The goal of this paper is to extend Girouard and Laugesen's upper bound to handle positive Robin parameters in the range $\alpha \in (0,4\pi]$, thus resolving one of their open problems \cite[p.{\,}2713]{GirLau2021}. Write $\D = \{ z\in\C : |z|<1 \}$ for the unit disk, so that $\D\sqcup\D$ represents a disjoint union of two copies of the disk. 
\begin{theorem}[Third Robin eigenvalue is maximal for two disks] \label{th:main}
Fix $\alpha\in[-4\pi,4\pi]$. If $\Omega\subset\R^2$ is a simply-connected bounded Lipschitz domain whose boundary is a Jordan curve then
\[
  \lambda_3(\Omega;\alpha/L) A <\lambda_3(\D\sqcup\D;\alpha/4\pi) 2\pi.
\]
Equality is attained asymptotically for the domain $\Omega_\e = (\D - 1 + \e)\cup (\D + 1 - \e)$ that as $\e \to 0$ approaches the disjoint union $(\D-1)\cup (\D+1)$ of two disks.
\end{theorem}
The third eigenvalue of the disjoint union $\D\sqcup\D$ is simply the second eigenvalue of one of the disks, and so the theorem says 
\[
\lambda_3(\Omega;\alpha/L) A < \lambda_2(\D;\alpha/4\pi) 2\pi ,
\]
where the eigenvalue on the right can be computed explicitly in terms of Bessel functions on the disk, as explained in \autoref{sec:diskproblem}. 

By scale invariance, the conclusion of the theorem can alternatively be rephrased as  
\[
\lambda_3(\Omega;\alpha/L(\Omega))<\lambda_3(\Omega^{\star\star};\alpha/L(\Omega^{\star\star})) 
\]
where $\Omega^{\star\star}$ is the union of two disjoint disks each having half the area of $\Omega$. 

As a remark, \autoref{th:main} can fail when $\alpha>65.4$, because $\lambda_2(\D;\alpha/4\pi) 2\pi<\alpha$ by \cite[pp.{\,}1039--1040]{FreiLauS2020} when $\alpha/2>32.7$, while by \cite[(4.1)]{FreiLauS2020} we have $\alpha \simeq \lambda_1(\Omega;\alpha/L) A \leq \lambda_2(\Omega;\alpha/L) A$ for any sufficiently long and thin rectangle $\Omega$. 

\subsection*{Plan of the paper}
The next section explains how to modify Girouard and Laugesen's approach in order to obtain \autoref{th:main} for $\alpha \in (0,4\pi]$. The key change is to use a degree theoretic lemma due to Karpukhin and Stern \cite[Lemma 4.2]{KS24} for mappings between spheres. Those authors wrote in reference to the theorem of Girouard and Laugesen for $\alpha \in [-4\pi,0]$ that ``We believe that our version of the argument allows one to extend the range of Robin parameters for which the results'' hold \cite[p.{\,}4079]{KS24}. The current paper  pursues their suggestion and shows that the crucial reflection-symmetry hypothesis \eqref{eq:reflectionsymmetry} indeed holds, enabling the proof to proceed. 

\autoref{sec:mobiuscaps} recalls properties of M\"{o}bius transformations and hyperbolic caps. From those objects the family of trial functions is constructed: see formula \eqref{eq:trialfn} and \autoref{fig:Trial}. 

To show in \autoref{sec:orthogonality} that at least one trial function in the family is orthogonal to the first two Robin eigenfunctions of $\Omega$, the reflection symmetry hypothesis for Karpukhin and Stern's degree theory lemma is verified. In \autoref{sec:degree} we present an alternative proof of Kim's variant of their lemma \cite{K24}. We simplify Kim's proof by introducing a homotopy to avoid certain calculations, thus bringing out more clearly the role played by the reflections.  

\autoref{sec:diskproblem} collects background facts on the Robin spectrum of the disk. 
        
\subsection*{Literature on upper bounds for eigenvalues of the Laplacian}

The maximization of individual eigenvalues of the Laplacian began with Szeg\H{o} \cite{S54}, who proved that among simply-connected planar domains of given area, the second Neumann ($\alpha=0$) eigenvalue is largest for the disk. Weinberger \cite{W56} extended the result by a different method to  domains in all dimensions. An excellent source for these classical results and later developments is the survey book edited by Henrot \cite{H17}. Notable open problems for Robin eigenvalues can also be found in Laugesen \cite{L19}. 

The Neumann inequalities of Szeg\H{o} and Weinberger were extended to the second Robin eigenvalue by Freitas and Laugesen \cite{FreiLauS2020,FreiLauW2021}. For the third Neumann eigenvalue, the breakthrough was achieved by Girouard, Nadirashvili and Polterovich \cite{GirNadPol21}, finding for simply connected planar domains that the maximizer is the union of two equal disks. Their result was generalized by Bucur and Henrot \cite{BH19} to all domains and higher dimensions. 

Meanwhile, a parallel line of research developed eigenvalue bounds on closed surfaces and manifolds, starting with Hersch's generalization of Szeg\H{o}'s approach to metrics on the $2$-sphere \cite{H70}, showing that the round sphere maximizes $\lambda_2$. The most relevant work for our current purposes is by Petrides \cite{P14}, Karpukhin and Stern \cite{KS24} and Kim \cite{K22,K24}. Petrides got upper bounds on $\lambda_3$ for spheres of arbitrary dimensions, and in the course of that work obtained an elegant degree theory lemma for ``reflection symmetric'' maps between spheres \cite[claim 3]{P14}. 

Karpukhin and Stern found a more general result \cite[Lemma 4.2]{KS24} involving two reflections rather than one. Kim applied Petrides's lemma in her work for eigenvalues on spheres \cite{K22} and applied a variant of Karpukhin and Stern's lemma to eigenvalues on projective space \cite{K24}. In the latter paper, she provided an alternative proof that yields more information about the value of the degree of the sphere mapping.

\section{\bf Proof of \autoref{th:main}}
\label{sec:mainproof}

The function spaces $L^2(\Omega;\C)$ and $H^1(\Omega;\C)$ consist of complex-valued functions, although for the sake of brevity we will often omit the $\C$ from the notation. A conformal map is a diffeomorphism that is holomorphic in both directions.

The Robin eigenvalues form a sequence $\lambda_1 \leq \lambda_2 \leq \lambda_3 \leq \cdots \to \infty$ in which each eigenvalue is repeated according to its multiplicity. The variational characterization of the third eigenvalue says that it equals the minimum of the Rayleigh quotient taken over trial functions orthogonal to the first two eigenfunctions:
\begin{align*}
	& \lambda_3(\Omega;\alpha/L) \\ 
	& =\min\left\{\frac{\int_{\Omega}|\nabla u|^2\,dA+(\alpha/L)\int_{\partial\Omega}|u|^2\,ds}{\int_{\Omega}|u|^2\,dA} \,:\,u\in H^1(\Omega; \C) \setminus \{ 0 \}, \int_{\Omega}u f_j\,dA=0, \ j=1,2 \right\} \notag
\end{align*}
where $f_1$ and $f_2$ are $L^2$-orthonormal real-valued eigenfunctions corresponding to the eigenvalues $\lambda_1(\Omega;\alpha/L)$ and $\lambda_2(\Omega;\alpha/L)$, and $dA$ is the area element. The trial function $u$ may be complex-valued. 

Thus to prove \autoref{th:main}, one wants to construct a suitable trial function $u$ and show that when it is substituted into the variational characterization, the desired upper bound is obtained on the third eigenvalue. 

The method of Girouard and Laugesen \cite[Section 7]{GirLau2021} for $\alpha \in [-4\pi,0]$ continues to hold verbatim for $\alpha \in [-4\pi,4\pi]$, except with one minor and one major alteration. The major change is that when $\alpha \in (0,4\pi]$, a new method is needed to prove existence of a trial function $u$ orthogonal to the Robin eigenfunctions $f_1$ and $f_2$. Girouard and Laugesen take a two-step approach to proving orthogonality. In terms of the notation developed in \autoref{sec:mobiuscaps} below, their first step shows for each each hyperbolic cap $C$ that a unique M\"{o}bius parameter $w$ exists that ensures orthogonality of $u$ against $f_1$ (their Lemma 5.1 and (5.3)). Second, they show for some $C$ that orthogonality also holds against $f_2$ (their Proposition 5.5). The hypothesis of their first step fails when $\alpha>0$ because the second Robin eigenfunction of a disk has nonmonotonic radial part $g$; see \autoref{Robin_g_second} in \autoref{sec:diskproblem}. 

To replace the orthogonality component of their argument, in the current paper we rely on \autoref{vanish2} below to generate a trial function that is orthogonal to $f_1$ and $f_2$. The proof of that proposition follows the one-step approach of Karpukhin and Stern. 

The proposition generates a trial function $u_{w,C}$ that depends on a point $w \in \D$ and a parameter $t \in [0,1]$. When $0 \leq t < 1$, Case 2 in the proof of Girouard and Laugesen \cite[Sections 7]{GirLau2021} should be followed. When $t=1$, one follows Case 1 with the minor change that $v \circ B$ in their proof is replaced by $v \circ M_w \circ B$. 

The final claim in the theorem, about asymptotic equality as $\e \to 0$, was handled for all $\alpha$ by Girouard and Laugesen \cite[Section 8]{GirLau2021}. \qed

\subsubsection*{Remark} The orthogonality construction in \autoref{vanish2} holds for all $\alpha$. The restriction $\alpha \geq -4\pi$ in \autoref{th:main} is necessary elsewhere in the proof of the theorem, on \cite[pp.{\,}2732-2733]{GirLau2021}, to ensure nonnegativity of $\lambda_2(\D;\alpha/4\pi)$ and hence to justify a certain inequality in the proof. The restriction $\alpha \leq 4\pi$ arises on \cite[pp.{\,}2732--2733]{GirLau2021} when Girouard and Laugesen use their Lemmas 4.3 and 6.1 to compare the $L^2$ norm of the trial function on $\Omega$ with the norm of the eigenfunction on $\D$. Those lemmas, which ultimately rely on \cite[formula (7.5)]{FreiLauS2020}, both require $\alpha \leq 4\pi$.

\section{\bf M\"obius transformations, hyperbolic caps, fold and cap maps, and trial functions}
\label{sec:mobiuscaps}
		
We follow Girouard and Laugesen's construction in \cite{GirLau2021} of a $4$-parameter family of complex-valued trial functions for the Rayleigh quotient of $\lambda_3(\Omega;\alpha/L)$. The four parameters will provide enough degrees of freedom to ensure at least one of the trial functions is orthogonal to the first two real-valued eigenfunctions $f_1$ and $f_2$. 
	
Two of the four parameters come from a family of M\"{o}bius transformations of the disk and two more from a family of hyperbolic caps inside the disk. We proceed to recall the needed formulas and reflection properties. 
	
\subsection*{M\"obius transformations}
Given $w \in \overline{\D}$, let
\[
M_w(z) = \frac{z+w}{z\overline{w}+1} , \qquad z \in \D .
\]
When $w\in\D$, this $M_w$ is a M\"{o}bius self-map of the disk with $M_w(0)=w$. When $w \in \partial\D$ it is a constant map, with $M_w(z)=w$ for all $z\in\D$. 

Define $R_p:\C\to\C$ to be reflection across the line through the origin that is perpendicular to $p \in \Sp^1 = \partial \D$. The following conjugation relation comes from \cite[formula (2.4)]{GirLau2021}:  
\begin{equation} \label{eq:Mobiusconjgen}
M_{R_p(w)} = R_p \circ M_w \circ R_p , \qquad w \in \D .
\end{equation}

\subsection*{Hyperbolic caps }
For $p \in \Sp^1$, let $C_{p,0}$ be the half-disk ``centered'' in direction $p$:
$$C_{p,0}=\{z\in \overline{\D} \,:\,z\cdot p\geq 0\},$$
where in this definition we think of $z$ and $p$ as vectors in $\R^2$. Define a hyperbolic cap $C=C_{p,t}\subset \overline{\D}$ by
\[
C_{p,t}=M_{-pt}(C_{p,0}), \qquad p \in \Sp^1 , \quad t \in (-1,1) ,
\]
 as illustrated in \autoref{fig:capdef}. The complement of $C$ is the cap $C^\star=C_{-p,-t}$.  
\begin{figure}
\begin{center}
\begin{tikzpicture}[scale=2.0]

\draw (0,0) circle (1.0) node[right] {\hspace*{2pt}\large $0$};

\filldraw  (-0.55,0.45) circle (0pt) node[right] {\huge $C$};

\filldraw  (0,0) circle (0.8pt); 
\filldraw  (0.7071,0.7071) circle (0.8pt) node[above] {\large \hspace*{12pt}$p$};
\filldraw  (0.7071-1,0.7071-1) circle (0.8pt) node[below] {\large $-pt$\hspace*{18pt}};

\draw[dotted,thick] (0.7071,0.7071) -- (0.7071-1.0,0.7071-1.0);

\draw[red,thick,domain=0:90] plot ({cos(\x)-1}, {sin(\x)-1});

\def\radelta{0.075}
\draw (-\radelta,-1) -- (-\radelta,-1+\radelta) -- (-0.1*\radelta,-1+\radelta);
\draw (-1,-\radelta) -- (-1+\radelta,-\radelta) -- (-1+\radelta,-0.1*\radelta);

\end{tikzpicture}
\end{center}
\caption{\label{fig:capdef} The hyperbolic cap $C=C_{p,t}$ is the image of the half-disk $C_{p,0}$ under the M\"obius transform $M_{-pt}$. As shown in the diagram, positive $t$ values correspond to caps larger than a half-disk.}
\end{figure}
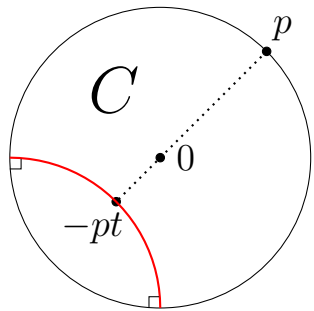

The hyperbolic reflection $\tau_C=\tau_{p,t}: \overline{\D} \to \overline{\D}$ associated with $C$ is obtained by pulling back to the half-disk, reflecting, and then pushing out again:  
\[
\tau_{p,t}=M_{-pt}\circ R_p\circ M_{pt}.
\]
This reflection $\tau_C$ takes $C$ to $C^\star$ and vice versa. 

This paper uses only the caps $C_{p,t}$ with $t \in [0,1)$, that is, caps larger than a half-disk. It will be important later that $C_{p,t}$ expands to the full disk as $t\to 1$.

\subsection*{Fold map}
Define the ``fold map'' $F_C:\overline{\D}\to C$ by
\[
F_C(z) =
\begin{cases}
z & \text{if\ } z \in C, \\
\tau_C(z) & \text{if\ } z \in C^\star .
\end{cases}
\]
We interpret $F_C$ as ``folding'' the disk onto the cap $C$ across the boundary between $C$ and its complement $C^\star$. The fold map is two-to-one except on the common boundary. Clearly $F_C(z)$ depends continuously on the parameters $(p,t,z)\in \Sp^1\times [0,1)\times \D$. 		

\subsection*{Cap map}
Girouard and Laugesen \cite[Section 3]{GirLau2021} constructed a particular conformal map $G_C:C \to \D$, called a ``cap map'', such that $G_C$ converges locally uniformly to the identity as $t \to 1$, that is, as the cap $C$ expands to fill the whole disk. 

\subsection*{Trial functions}
Let $\alpha \in \R$. For the rest of the paper, 
\[
v=g(r)e^{i\theta}
\]
is the complex-valued Robin eigenfunction on the unit disk corresponding to $\lambda_2(\D;\alpha/4\pi)$; see \autoref{basic2}, where it is observed that $g(1)>0$. The Robin parameter here is $\alpha/4\pi$. 

Fix a conformal map $B : \Omega \to \D$. 
%
%
Given a hyperbolic cap $C$ and $w \in \overline{\D}$, define the trial function 
\[
u_{w,C} : \Omega \to \C
\]
by
\begin{equation} \label{eq:trialfn}
u_{w,C}=v\circ M_w \circ G_C \circ F_C \circ B,
\end{equation}
as shown schematically in \autoref{fig:Trial}. This function $u_{w,C}(z)$ is continuous as a function of $z \in \Omega$, is bounded by the maximum value of the radial part $|v|=g$, is smooth except along the preimage under $B^{-1}$ of the cap boundary, and belongs to $H^1(\Omega)$ by conformal invariance of the Dirichlet integral. 

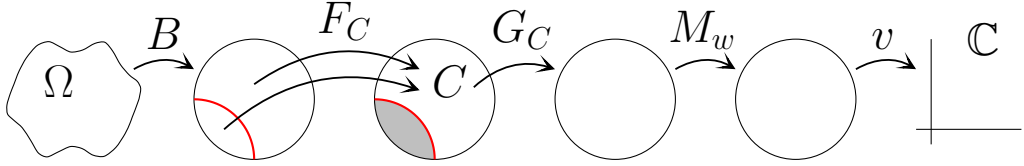
\begin{figure}
\begin{tikzpicture}[scale=0.8]

\def\omegaoffset{-6}
\def\omegasize{0.62}
\draw plot[smooth,samples=36,domain=0:360,variable=\t] ({\omegaoffset+\omegasize*0.35*(1+0.05*sin(6*\t))*(5*cos(\t))},{\omegasize*1.5*(1+0.08*cos(6*\t)+0.08*cos(6*\t))*(sin(\t))});

 \begin{scope}
    \clip (0,0) circle (1.0);
    \draw[draw=none,fill=gray!50] (-1,-1) circle (1.0);
\end{scope}

\draw (-3,0) circle (1.0);
\draw (0,0) circle (1.0);
\draw (3,0) circle (1.0);
\draw (6,0) circle (1.0);

\draw[red,thick,domain=0:90] plot ({cos(\x)-1}, {sin(\x)-1});
\draw[red,thick,domain=0:90] plot ({cos(\x)-4}, {sin(\x)-1});

\draw[-] (8,-0.5) -- (9.75,-0.5);
\draw[-] (8.25,-0.75) -- (8.25,1) node[right] {\Large \hspace*{4pt} $\mathbb C$};

\myarrowL{(-5,0.5)}{(-4,0.5)};
\myarrowL{(-3,0.25)}{(-0.25,0.5)};
\myarrowL{(-3.5,-0.5)}{(-0.25,0.15)};
\myarrowL{(0.65,0.2)}{(2,0.5)};
\myarrowL{(4,0.5)}{(5,0.5)};
\myarrowL{(7,0.5)}{(8,0.5)};

\filldraw  (-4.5,0.65) circle (0pt) node[above] {\Large $B$};
\filldraw  (-1.5,0.8) circle (0pt) node[above] {\Large $F_C$};
\filldraw  (1.45,0.65) circle (0pt) node[above] {\Large $G_C$};
\filldraw  (4.45,0.65) circle (0pt) node[above] {\Large $M_w$};
\filldraw  (7.45,0.65) circle (0pt) node[above] {\Large $v$};

\filldraw  (-6.25,0.3) circle (0pt) node {\Large $\Omega$};
\filldraw  (0.25,0.3) circle (0pt) node {\Large $C$};

\end{tikzpicture}

\caption{\label{fig:Trial} The trial function $u_{w,C}$ on $\Omega$ is constructed by precomposing the (complex-valued) second Robin eigenfunction $v$ on the disk with four transformations: map conformally from $\Omega$ to the disk, fold onto the cap $C$, expand the cap to the whole disk, apply a M\"{o}bius map of the disk, and finally evaluate $v$. The cap $C$ and M\"{o}bius parameter $w$ will be chosen to ensure orthogonality of the trial function to the first and second Robin eigenfunctions on $\Omega$.}
\end{figure}

Further, the trial function depends continuously on its parameters:
\begin{lemma}[Continuous dependence of trial function; \protect{\cite[Lemma 5.2]{GirLau2021}}] \label{lemma:continuousTrialFunction}
The function $u_{w,C}(z)$ depends continuously on $(C,w,z)$, in other words, on $(p,t,w,z)$. 
\end{lemma}
As the cap expands to the whole disk, the fold and cap maps drop out of the formula completely, yielding a limiting value at $t=1$ that is independent of the cap direction $q$: 
\begin{lemma}[Extension of trial function to large caps; \protect{\cite[Lemma 5.3]{GirLau2021}}] \label{lemma:extension}
The function $u_{w,C}(z)$ with $C=C_{p,t}$ extends to $t=1$ as follows: 
\[
    u_{w,C} \to v \circ M_{\widetilde{w}} \circ B \qquad \text{as $w \to \widetilde{w} \in \overline{\D}$, \ $p \to q \in \Sp^1, \ t \to 1$,} 
\]
with locally uniform convergence on $\Omega$. 
\end{lemma}

\section{\bf Orthogonality of trial functions}
\label{sec:orthogonality}

Denote by $f_1$ and $f_2$ the first and second eigenfunctions (real-valued) of the Robin Laplacian on $\Omega$, corresponding to eigenvalues $\lambda_1(\Omega;\alpha/L)$ and $\lambda_2(\Omega;\alpha/L)$, where $\alpha \in \R$. As explained in \autoref{sec:mainproof}, the task for proving \autoref{th:main} is to show that for some $C=C_{p,t}$ and M\"{o}bius parameter $w$, the trial function $u_{w,C}$ constructed in the previous section is $L^2$-orthogonal to $f_1$ and $f_2$. In \autoref{vanish2} below, we establish the desired orthogonality. 

Since the ground state $f_1$ does not change sign, we may suppose it is positive. Define 
\[
f_* = f_2-\rho f_1 \qquad \text{where} \quad \rho = \frac{\int_{\Omega}  f_2 \,dA }{ \int_{\Omega} f_1 \,dA} ,
\]
so that $\int_\Omega f_* \, dA = 0$. Introduce a continuous vector field 
\[
V : \overline{\D} \times \Sp^{1} \times[0,1) \to \C \times \C \simeq \R^4
\]
that is defined by
\begin{align*}
& V(w,p,t) \\
& = \left( \, \langle u_{w,C} , f_1 \rangle_{L^2} \, , \langle u_{w,C} , f_* \rangle_{L^2} \, \right) \\
& = \left( \int_\Omega (v \circ M_w \circ G_C \circ F_C \circ B)(z) f_1(z) \, dA \, , \int_\Omega (v \circ M_w \circ G_C \circ F_C \circ B)(z) f_*(z) \, dA \right) .
\end{align*}
The vector field extends continuously to $t=1$ with value 
\begin{equation} \label{Vt1}
V(w, p,1) = \left( \int_\Omega (v \circ M_w \circ B)(z) f_1(z) \, dA \, , \int_\Omega (v \circ M_w \circ B)(z) f_*(z) \, dA \right) 
\end{equation}
by \autoref{lemma:extension} and dominated convergence, using that $u_{w,C}$ is dominated by the maximum value of $g$. Notice this value $V(w,p,1)$ in \eqref{Vt1} is independent of $p$ and depends only on $w$.   

Another useful property is that when $w = e^{i\theta} \in \partial \D$ the value of $V$ is nonzero and independent of $p$ and $t$, with 
\begin{equation} \label{Vboundary}
V(e^{i\theta}, p,t) = \left(  g(1) e^{i\theta} \! \int_\Omega f_1 \, dA \, , 0 \right) \neq (0,0)
\end{equation}
since $M_{e^{i\theta}}(\zeta) = e^{i\theta}$ for all $\zeta \in \D$ while $v(e^{i\theta})=g(1) e^{i\theta}$ with $g(1)>0$, and $\int_\Omega f_1 \, dA > 0$ and $\int_\Omega f_* \, dA = 0$. 

The goal is to show the vector field vanishes at some $(w,p,t)$, because $V(w,p,t)=0$ implies that $u_{w,C}$ is orthogonal in $L^2(\Omega)$ to $f_1$ and $f_*$ and hence also to $f_2$. The next proposition uses homotopy invariance of degree to show that $V$ indeed vanishes somewhere. 
\begin{proposition}[Vanishing of the vector field]\label{vanish2}
$V(w,p,t)=0$ for some $(w,p) \in \overline{\D} \times \Sp^{1}$ and $t \in [0,1]$. 
\end{proposition}
The $w$ provided by the proposition cannot lie in $\partial \D$, due to \eqref{Vboundary}, and so $w \in \D$.  
\begin{proof}
We start by endowing an equivalence relation on the parameter space: say that $(w,p) \sim (w, q)$ if $w \in \partial \overline{\D}$ and $p,q \in \Sp^{1}$. Then define a bijection to the $3$-sphere by
\[
\begin{split}
\Psi : \big( (\overline{\D} \times \Sp^{1})/ \sim \big) & \to  \Sp^{3} \\
(w,p) & \mapsto (a,b) = (\sqrt{2-|w|^2} \, w, (1-|w|^2)p) 
\end{split}
\]
where it is straightforward to check that $|a|^2+|b|^2=1$. This mapping collapses the boundary points $w \in \partial \D$ onto points in $\Sp^{3}$ of the form $(a,0)$.  The inverse map has
\[
(w,p) = \Psi^{-1}(a,b) = \left( \frac{a}{\sqrt{1+|b|}}, \frac{b}{|b|} \right) 
\]
when $b \neq 0$, where $w$ really depends only on $a$ since $|b| =\sqrt{1-|a|^2}$. When $b=0$, we have $|a|=1$ and so $w=a \in \partial \D$, giving $\Psi^{-1}(a,0)=(a,p)$ with $p$ an arbitrary point in $\Sp^{1}$ (they are all equivalent under $\sim$). 

Next, write the components of the inverse map as 
\begin{equation} \label{eq:wa}
w(a) = \frac{a}{\sqrt{1+\sqrt{1-|a|^2}}} , \qquad p(b) = \frac{b}{|b|} . 
\end{equation}
Precompose the vector field $V$ with this inverse, letting 
\[
\widetilde{V}(a,b,t)=V(w(a), p(b), t), \qquad (a,b) \in \mathbb \Sp^{3}, \quad t \in [0,1] ,
\]
and noting when $b=0$ that the choice of $p(b)$ is irrelevant because then $w(a) =a = e^{i\theta} \in \partial \D$, in which case $\widetilde{V}(a,0,t)$ is independent of $p(b)$ by \eqref{Vboundary}. Thus $\widetilde{V}$ is continuous on $\Sp^3 \times [0,1]$. 

Suppose for the sake of obtaining a contradiction that $V(w,p,t)$ does not vanish. Then $\widetilde{V}$ also does not vanish. Normalize $\widetilde{V}$ by defining 
\[
W_t : \Sp^{3} \to  \Sp^{3}
\]
with  
\[
W_t(a,b)=\frac{\widetilde{V}(a,b,t)}{|\widetilde{V}(a,b,t)|} , \qquad t \in [0,1] .
\]
Formula \eqref{Vboundary} implies when $b=0$ that $W_t(a,0) = (a,0)$. 

We claim $W_1$ has degree $0$. Indeed, when $t=1$ we have 
\[
\widetilde{V}(a,b,1) = \left( \int_\Omega (v \circ M_{w(a)} \circ B)(z) f_1(z) \, dA \, , \int_\Omega (v \circ M_{w(a)} \circ B)(z) f_*(z) \, dA \right) 
\]
by \eqref{Vt1}, and this last expression is independent of $b$. Hence the continuous map $W_1 : \Sp^{3} \to  \Sp^{3}$ is not surjective, because it depends only on the $2$-dimensional parameter $a \in \overline{\D}$, is constant when $a \in \partial \D$, and is smooth when $a \in \D$. Since the range of $W_1$ omits some point of $\Sp^3$, we see $W_1$ is homotopic to a constant map and so has degree $0$. Homotopy invariance of degree implies that $W_0$ also has degree $0$. 

We will apply \autoref{topdeg} below to $W_0$ to obtain the desired contradiction, namely that $W_0$ has degree $1$. Hence we need only verify that $W_0$ satisfies the ``reflection symmetry'' hypothesis \autoref{refsym} in \autoref{topdeg}. In this task some shorthand notations are helpful, for the reflection, cap map and fold map associated with the half-disk: let $R_b=R_{p(b)}, G_b = G_{C_{p(b),0}}$ and $F_b=F_{C_{p(b),0}}$. Then for $b \neq 0$ and $t=0$, the definitions say that
\begin{align*}
&\widetilde{V}(R_b(a),-b,0) \\
&= \left( \int_{\Omega} (v \circ M_{w(R_b(a))} \circ G_{-b} \circ F_{-b} \circ B) f_1\, dA \, , \int_{\Omega} (v \circ M_{w(R_b(a))} \circ G_{-b} \circ F_{-b} \circ B) f_*\, dA \right) \\
&= \left( \int_{\Omega} (v \circ M_{R_b(w(a))} \circ R_b \circ G_b \circ F_b \circ B) f_1 \, dA \, , \int_{\Omega} (v \circ M_{R_b(w(a))} \circ R_b \circ G_b \circ F_b \circ B) f_*\, dA \right)  
\end{align*}
using $w(R_b(a)) = R_b(w(a))$ by \eqref{eq:wa} and linearity of $R_b$, and that $G_{-b}=R_b \circ G_b \circ R_b$ by \cite[formula (5.5)]{GirLau2021} and $R_b \circ F_{-b} = F_{b}$ by definition of the fold map. Since also $M_{R_b (w)} \circ R_b = R_b \circ M_w$ by \eqref{eq:Mobiusconjgen} and $v \circ R_b = R_b \circ v$ by \cite[formula (5.6)]{GirLau2021}, we deduce 
\begin{align}
& \widetilde{V}(R_b(a),-b,0) \notag \\
&= \left( R_b \! \int_{\Omega} (v \circ M_{w(a)} \circ G_b \circ F_b \circ B) f_1\, dA \, , R_b \! \int_{\Omega} (v \circ M_{w(a)} \circ G_b \circ F_b \circ B) f_*\, dA \right) \notag \\
&= (R_b \times R_b) \widetilde{V}(a,b,0). \label{eq:reflectionsymmetry}
\end{align}
The reflection $R_b$ preserves the magnitude, and so $|\widetilde{V}(R_b(a),-b,0)| = |\widetilde{V}(a,b,0)|$. Thus after dividing by these quantities, we see $W_0(R_b(a),-b) = (R_b \times R_b) W_0(a,b)$, which is the reflection symmetry condition \autoref{refsym} for $\phi=W_0$. 

Hence $W_0$ has degree $1$ by \autoref{topdeg}. This contradiction implies that $V(w,p,t)$ must vanish somewhere, completing the proof of \autoref{vanish2}.
\end{proof}

\section{\bf Degree theory result}
\label{sec:degree}

This section states and proves \autoref{topdeg}, which was used in the previous section and says that reflection-symmetric self-maps of $\mathbb{S}^3$ have degree $1$. 

The degree $\deg(\phi)$ of a continuous mapping $\phi$ from a sphere $\Sp^{N-1}$ to itself is defined to equal $d(\psi,\B^N,0)$, that is, the degree of $\psi$ at $0$ with respect to $\B^N$ where $\psi$ is any continuous extension of $\phi$ to the ball $\overline{\B^N}$; see \cite[Proposition 1.27]{FG95}.

Recall that $R_b : \R^2 \to \R^2$ is reflection across the line through the origin that is perpendicular to $b \in \R^2, b \neq 0$. Define 
\[
R_b \times R_b : \R^4 \to \R^4 
\]
by $(R_b \times R_b)(c,d) = (R_b(c),R_b(d))$ when $c,d \in \R^2$. 
\begin{theorem}[Degree of reflection-symmetric map between $3$-spheres] \label{topdeg}
Suppose $\phi: \Sp^{3} \to \Sp^{3}$ is continuous. If $\phi$ satisfies the reflection symmetry property
\begin{equation}
\label{refsym}
\begin{split}
(R_b \times R_b) \phi(a,b) &= \phi(R_b(a), -b) \qquad\text{when $b \neq 0$,} \\
\phi(a,0) &= (a,0) \qquad\qquad\quad\, \text{when $b = 0$,}
\end{split}
\end{equation}
for all $a, b \in \R^{2}$ with $(a,b) \in \Sp^{3}$, then $\textnormal{deg}(\phi) =1$.
\end{theorem}
The reflection symmetry condition \eqref{refsym} can be interpreted as saying that $R_b \times R_b$ commutes with $\phi$, since $R_b(b)=-b$. But one must remember that the reflection operator $R_b$ depends on $b$ and so varies from point to point. 

A recent paper by Kim \cite[Theorem 11]{K24} shows that the degree of a map with reflection symmetry on $\Sp^{2n+1}$ equals $1$ for odd $n$ (in particular for $n=1$ as in the theorem here) and is odd for even $n$. The current paper improves the method by using a homotopy to avoid certain computations needed in \cite{K24}: see $\varphi_t$ later in the section, in the proof of \autoref{changedeg}. 

Karpukhin and Stern \cite[Lemma 4.2]{KS24} earlier showed that a continuous map on on $\Sp^{2n+1}$ with a similar reflection property has odd degree, by using the Lefschetz--Hopf fixed point theorem. The current paper builds on their remark that their lemma with $n=1$ should apply to the Robin problem. 

The special case $\Sp^3$ treated by \autoref{topdeg} is more readily visualized than the general case. We illustrate the essential ideas in \autoref{fig:phi2} and \autoref{fig:phi3} below. 
\begin{proof}
We aim to extend $\phi: \Sp^{3} \to \Sp^{3}$ to a continuous mapping $\phi_3 : \overline{\B^4} \to \R^4$, so that the degree of $\phi$ can be computed from the degree of $\phi_3$. The extension is accomplished in stages, building up from $3$ dimensions to $4$. For that purpose, we write  
\[
\R^3=\{ x \in \R^4 : x_4=0 \} \qquad \text{and} \qquad \B^3=\{ x \in \B^4 : x_4 = 0 \} .
\]
Similarly, $\B^2$ means $\{ x \in \B^4 : x_3 = x_4 = 0 \}$. 

Start by letting $\phi$ be the identity mapping on $\B^2$, that is, $\phi(a,0)=(a,0)$ whenever $a \in \B^2$, which is consistent with the second equation in \eqref{refsym}. Also let $\phi$ be the identity on the $3$-dimensional upper halfball $\overline{\B^3_+(1/2)}$ of radius $1/2$. (We use ``$+$'' subscripts to refer to the upper half of an object and ``$-$'' for the lower half.) In particular, $\phi(0)=0$. 

\smallskip\noindent 
\boxed{\text{Step 1 --- extend to upper half $3$-ball.}} 
In this step and the next, we aim to extend $\phi$ continuously to the $3$-dimensional ball $\overline{\B^3}$ in such a way that the extension still satisfies the reflection symmetry equation and vanishes only at the origin. We start by extending to the upper half of the ball. 

Consider the annular region $A = \B^3 \setminus \overline{\B^3(1/2)}$ and its upper half 
\[
A_+ = \B^3_+ \setminus \overline{\B^3_+(1/2)} .
\]
The boundary of $A_+$ (as a domain in $\R^3$) decomposes into three parts: $\Sp^2_+, \Sp^2_+(1/2)$, and $\{ (a,0) : a \in \overline{\B^2}, 1/2 \leq  |a| \leq 1 \}$, which we call the outer, inner, and bottom boundaries, respectively. We have already chosen $\phi$ to equal the identity on the inner and bottom boundaries, and since by hypothesis $\phi$ equals the identity on the unit circle where the outer boundary meets the inner one, we see that $\phi$ is continuous on the whole boundary $\partial A_+$. Remember also that $|\phi|=1$ on the outer boundary. 

Take $\phi_1 : \overline{\B^3_+} \to \R^4$ to be any continuous extension of $\phi$ that is smooth on $A_+$. For example, we could choose $\phi_1$ on $A_+$ to be the harmonic extension of the boundary values $\phi |_{\partial A_+}$, meaning the we extend each of the four components of $\phi$ harmonically. Since $\phi$ does not vanish on $\partial A_+$, we know the boundary image $\phi_1(\partial A_+)$ does not contain the origin. Hence $|\phi_1(\cdot)|$ attains a positive minimum value on $\partial A_+$, and so for some sufficiently small $\e>0$, the image $\phi_1(\partial A_+)$  does not intersect the ball $\overline{\B^4(\e)}$. Thus the preimage $\phi_1^{-1}(\overline{\B^4(\e)})$ is separated from the boundary, meaning there exists $\delta>0$ such that 
\begin{equation} \label{epsdelta}
x \in \overline{A_+} \quad \text{with}\quad \phi_1(x) \in \overline{\B^4(\e)} \quad \Longrightarrow \quad \dist(x,\partial A_+) \geq \delta > 0 .
\end{equation}

\smallskip\noindent 
\boxed{\text{Step 2 --- eliminate zeros and extend to lower half $3$-ball.}} 
The map $\phi_1$ might have zeros in $A_+$. To eliminate them, we perturb the map as follows and then extend to $\B^3$. See \autoref{fig:phi2}.
\begin{figure}
\includegraphics[scale=1]{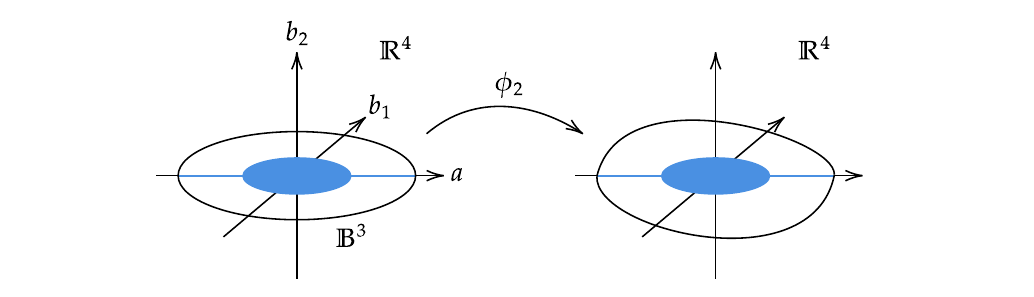}
\caption{\label{fig:phi2} Steps 1 and 2 --- extension of $\phi$ to $\phi_2$. We consider the mapping $\phi$ on the sphere $\Sp^2$ (shown here as a circle) and extend to the ball $\B^3$, mapping it into $\R^4$ while maintaining reflection symmetry. The extended map $\phi_2$ equals the identity on the sets shown in blue, namely the $3$-ball of radius $1/2$ and the lower dimensional ball $\B^2$ represented along the horizontal axis. The image of $\B^3$ generally lies outside $\R^3$, as indicated by the deformed circle on the right.}
\end{figure}

The $3$-dimensional domain $A_+$ is mapped smoothly under $\phi_1$ into $\R^4$. The image $\phi_1(A_+)$ has zero $4$-dimensional volume, and so does not contain the ball $\overline{\B^4(\e)}$. Thus we may take a point $z \in \overline{\B^4(\e)} \setminus \phi_1(A_+)$. Define a perturbed mapping $\phi_2$ on $\overline{\B^3_+}$ by 
\[
\phi_2(x) = 
\begin{cases}
\phi_1(x) & \text{when $x \in \overline{\B^3_+} \setminus A_+$,} \\
\phi_1(x) - z \dist(x,\partial A_+)/\delta & \text{when $x \in A_+$ with $\dist(x,\partial A_+)<\delta$,} \\
\phi_1(x) - z & \text{when $x \in A_+$ with $\dist(x,\partial A_+) \geq \delta$.}
\end{cases}
\] 
This perturbed mapping is continuous, since the distance function is continuous and equals $0$ on $\partial A_+$. Further, $\phi_2$ equals $\phi$ on $\overline{\B^3_+} \setminus A_+$, and in particular equals the identity on the upper halfball $\overline{\B^3_+(1/2)}$.

This construction ensures that $\phi_2(x) \neq 0$ when $x \in A_+$, as we now show. The claim is immediate if $\dist(x,\partial A_+) \geq \delta$ since then $\phi_2(x) = \phi_1(x) - z$, which is nonzero because $z \notin \phi_1(A_+)$. And if $\dist(x,\partial A_+) < \delta$ then $z \dist(x,\partial A_+)/\delta \in \B^4(\e)$, while $\phi_1(x) \notin \overline{\B^4(\e)}$ by \eqref{epsdelta}, and so 
\[
\phi_2(x) = \phi_1(x) - z \dist(x,\partial A_+)/\delta \neq 0 .
\]

Now that we have extended $\phi$ to $\phi_2$ on the upper halfball in $3$ dimensions, we may extend to the lower halfball by defining
\begin{equation} \label{phi2extended}
\phi_2(a,b) = (R_b \times R_b) \phi_2(R_b(a),-b) 
\end{equation}
when $(a,b) \in \overline{\B^3_-} \setminus \overline{\B^2}$, that is, when $(a,b) \in \overline{\B^4}, a=(a_1,a_2)$, and $b=(b_1,0)$ with $b_1 < 0$. This definition ensures that $\phi_2$ satisfies the reflection symmetry property \eqref{refsym} for all $(a,b) \in \overline{\B^3}$ with $b_1 \neq 0$ (equivalently, with $b \neq 0$). Also, $\phi_2$ equals the identity mapping on the lower halfball of radius $1/2$, because if $(a,b) \in \overline{\B^3_-(1/2)}$ then the definition gives that 
\[
\phi_2(a,b) = (R_b \times R_b) \phi_2(R_b(a),-b) = (R_b \times R_b) (R_b(a),-b) = (a,b) ,
\]
since $\phi_2$ equals the identity on the upper halfball of radius $1/2$. Similarly, $\phi_2=\phi$ on the lower half-sphere $\Sp^2_-$, since $\phi_2=\phi$ on the upper half-sphere and both $\phi_2$ and $\phi$ satisfy reflection symmetry. 

We must verify that the extended map $\phi_2$ is continuous on the set $\overline{\B^2}$ where the upper half of the $3$-ball joins the lower half. Since $b=(b_1,0)$, the reflection $R_b$ acts on $\R^2$ simply by changing the sign of the first coordinate. Thus the extended definition \eqref{phi2extended} says 
\[
\phi_2(a,b) = \Xi\big(\phi_2(-a_1,a_2,-b_1,0)\big)
\]
where 
\[
\Xi(y_1,y_2,y_3,y_4)=(-y_1,y_2,-y_3,y_4)
\]
flips the sign of the first and third coordinates in $\R^4$. Suppose $(a^\prime,b^\prime) \in \overline{\B^4}$ with $b_2^\prime=0$. If $b_1^\prime \nearrow 0$ and $a^\prime \to a \in \overline{\B^2}$ then 
\begin{align*}
\lim \phi_2(a^\prime,b^\prime) 
& = \Xi\big(\phi_2(-a_1,a_2,0,0)\big) \\
& = \Xi (-a_1,a_2,0,0) \qquad \text{since $\phi_2=\phi=\text{identity}$ on $\overline{\B^2}$} \\
& = (a_1,a_2,0,0)=\phi_2(a,0) .
\end{align*}
Thus $\phi_2$ is continuous on the joining set $\overline{\B^2}$, and so on all of $\overline{\B^3}$.

Step 2 is complete: we have found a continuous mapping $\phi_2$ of $\overline{\B^3}$ into $\R^4$ that satisfies there the reflection symmetry property \eqref{refsym}, agrees with $\phi$ on the sphere $\Sp^2$, equals the identity on $\overline{\B^3(1/2)}$, and is nonzero except at the origin. 

\smallskip\noindent 
\boxed{\text{Step 3 --- extend to $4$-ball.}} 
The next task is to extend $\phi_2$ continuously to the $4$-dimensional ball $\overline{\B^4}$ while preserving reflection symmetry and ensuring the extension equals $\phi$ on $\Sp^3$. \autoref{fig:phi3} summarizes the extension process. 
\begin{figure}
\includegraphics[scale=0.9]{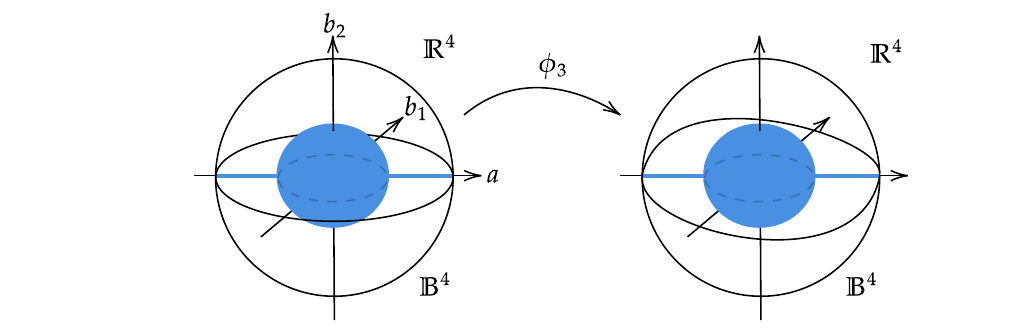}
\caption{\label{fig:phi3} Step 3 --- extension of $\phi_2$ to $\phi_3$. We consider the mapping $\phi_2$ on the sphere $\Sp^3$ and extend to the ball $\B^4$ while maintaining reflection symmetry. The extended map $\phi_3$ is the identity on the $4$-ball of radius $1/2$, shown in blue, and on the lower dimensional ball $\B^2$ represented along the horizontal axis.}
\end{figure}

First extend $\phi_2$ to be the identity map on the upper halfball $\overline{\B^4_+(1/2)}$ of radius $1/2$, and define $\phi_2$ to equal $\phi$ on $\Sp^3_+$. Consider the annular region 
\[
\mathcal{A} = \B^4 \setminus \overline{\B^4(1/2)} 
\] 
and its upper and lower halves 
\[
\mathcal{A}_+ = \B^4_+ \setminus \overline{\B^4_+(1/2)} , \qquad \mathcal{A}_- = \B^4_- \setminus \overline{\B^4_-(1/2)} .
\]

Choose $\phi_3 : \overline{\B^4_+} \to \R^4$ to be a continuous extension of $\phi_2$ to the upper halfball. For example, we could harmonically extend the components of $\phi_2$ from $\partial \mathcal{A}_+$ into the interior of $\mathcal{A}_+$. Further extend $\phi_3$ to the lower halfball $\B^4_-$ (which contains $\mathcal{A}_-$) by defining
\begin{equation} \label{refsym3}
\phi_3(a,b) = (R_b \times R_b) \phi_3(R_b(a),-b), \qquad (a,b) \in \B^4_- .
\end{equation}
This definition ensures that $\phi_3$ satisfies the reflection symmetry property \eqref{refsym} on the ball $\overline{\B^4}$. (We already knew that property on $\overline{\B^3}$, by the previous step.) Further, $\phi_3$ equals the identity mapping on the lower halfball of radius $1/2$, because if $(a,b) \in \overline{\B^4_-(1/2)}$ then 
\[
\phi_3(a,b) = (R_b \times R_b) \phi_3(R_b(a),-b) = (R_b \times R_b) (R_b(a),-b) = (a,b) ,
\]
since $\phi_3$ is the identity on $\overline{\B^4_+(1/2)}$. Similarly, reflection symmetry for $\phi_3$ and $\phi$ guarantees that $\phi_3=\phi$ on the lower half-sphere $\Sp^3_-$, since $\phi_3=\phi$ on the upper half-sphere by construction. 

We still need to verify continuity of $\phi_3$ on the set $\overline{\B^3}$ where the upper half of $\overline{\B^4}$ meets the lower half. Suppose $(a^\prime,b^\prime) \in \overline{\B^4_-} \setminus \overline{\B^3}$ with $(a^\prime,b^\prime) \to (a,b) \in \overline{\B^3}$ and $b_2^\prime \nearrow b_2=0$. We consider two cases. First, if $b_1 \neq 0$ then $b \neq 0$ and so the definition gives 
\begin{align*}
\lim \phi_3(a^\prime,b^\prime) 
& = \lim (R_{b^\prime} \times R_{b^\prime}) \phi_3 (R_{b^\prime}(a^\prime),-b^\prime) \\
& = (R_{b} \times R_{b}) \phi_3 (R_{b}(a),-b) \\
& = \phi_3(a,b) 
\end{align*}
by reflection symmetry for $\phi_3$. Second, if $b_1=0$ then $b=0$ and so 
\[
\phi_3(a,b)=\phi_3(a,0)=\phi(a,0)=(a,0)
\]
since $\phi$ is the identity on $\overline{\B^2}$ by construction. Hence
\begin{align*}
& \lim |\phi_3(a^\prime,b^\prime) - \phi_3(a,b)| \\
& = \lim |(R_{b^\prime} \times R_{b^\prime}) \phi_3(R_{b^\prime}(a^\prime),-b^\prime)-(a,0)| && \text{by \eqref{refsym3}} \\
&=\lim |\phi_3(R_{b^\prime}(a^\prime),-b^\prime)-(R_{b^\prime}(a^\prime),0)| && \text{since reflections preserve norms} \\
&=\lim |\phi_3(R_{b^\prime}(a^\prime),-b^\prime)-\phi_3(R_{b^\prime}(a^\prime),0)| && \text{since $\phi_3$ is the identity on $\overline{\B^2}$} \\
&= 0
\end{align*}
by uniform continuity of $\phi_3$ on $\overline{\B^4_+}$ (recall $-b_2^\prime>0$). Thus in both cases, $\phi_3$ is continuous at $(a,b)$.

\smallskip\noindent 
\boxed{\text{Step 4 --- compute degree by decomposition of $4$-ball.}} The ball has disjoint decomposition 
\[
\B^4 = \B^4(1/2) \cup \mathcal{A}_+ \cup \mathcal{A}_- \cup K
\] 
where
\[
K=\partial \B^4(1/2) \cup \big( \B^3 \setminus \overline{\B^3(1/2)} \big) .
\]
Notice $\phi_3$ is nonzero on $K$ by construction in Steps 2 and 3. Removing $K$ from the ball preserves the degree, by the excision property \cite[Theorem 2.7]{FG95}, and so $d(\phi_3, \B^{4},0) = d(\phi_3, \B^{4} \setminus K,0)$. Then the domain decomposition property of degree \cite[Theorem 2.7]{FG95} implies  
\begin{align*}
d(\phi_3, \B^{4},0) 
& = d(\phi_3,\B^4(1/2),0 )+d(\phi_3,\mathcal{A}_+,0 )+d(\phi_3,\mathcal{A}_-,0 ) \\
& = d(\phi_3,\B^4(1/2),0 ) \qquad \quad \text{by \autoref{changedeg} below} \\
& = 1 
\end{align*}
since $\phi_3$ equals the identity on $\B^4(1/2)$. Hence $\text{deg}(\phi, \mathbb{S}^3)=d(\phi_3, \B^{4},0 )=1$, which finishes the proof of \autoref{topdeg}. 
\end{proof}
The next lemma was used above in the proof of \autoref{topdeg}. It is a special case of \cite[Lemma 17]{K24} (take $n=1$ there). Several computations used in the proof in \cite{K24} are eliminated in the new and shorter proof below, by introducing a deformation $\varphi_t$.  
\begin{lemma}[Reflection symmetry and degrees on the half-annuli]
\label{changedeg}
Let $\mathcal{A}_{+}$ and $\mathcal{A}_{-}$ be the upper and lower annular regions of $\B^4 \setminus \overline{\B^4(1/2)}$.  If $\varphi: \overline{\mathcal{A}_{+} \cup \mathcal{A}_{-}} \to \R^{4}$ is a continuous map satisfying the reflection symmetry property \autoref{refsym} and $\varphi$ is nonzero on $\partial \mathcal{A}_{+} \cup \partial \mathcal{A}_{-}$ then the degrees on the half-annuli sum to zero:  
\begin{align*}
d(\varphi, \mathcal{A}_{+}, 0) + d(\varphi, \mathcal{A}_{-}, 0) = 0.
\end{align*}
\end{lemma}
\begin{proof}
Since $\varphi$ does not vanish where $b_2=0$ (note such points lie on the boundary of $\mathcal{A}_+$), we may choose $\delta \in (0,1)$ sufficiently small that $\varphi$ does not vanish when $|b_2| \leq \delta$. The excision property of degree allows us to truncate the half-annuli without changing the degrees, namely omitting all points with $|b_2| \leq \delta$ and concluding that $d(\varphi, \mathcal{A}_{\pm}, 0)=d(\varphi, \mathcal{A}_{\pm \delta}, 0)$ where 
\[
\mathcal{A}_{+\delta} = \{ (a,b) \in \mathcal{A}_+ : b_2>\delta \} , \qquad \mathcal{A}_{-\delta} = \{ (a,b) \in \mathcal{A}_- : b_2<-\delta \} .
\]
The goal is now to show $d(\varphi, \mathcal{A}_{+\delta}, 0) + d(\varphi, \mathcal{A}_{-\delta}, 0) = 0$.

On $\overline{\mathcal{A}_{+\delta}}$ we have $b_2 \geq \delta$ and so $b \neq 0$, and hence 
\[
\varphi(a,b) = (R_b \times R_b) \varphi(R_b(a), -b) 
\]
by the reflection symmetry property \autoref{refsym}. Let 
\[
\beta(t,b) = (1-t)b+t(0,1) \in \R^2 , \qquad t \in [0,1] ,
\]
so that $\beta(0,b)=b$ and $\beta(1,b)=(0,1)$. Clearly $\beta(t,b) \neq 0$ because its second coordinate $(1-t)b_2+t$ is greater than or equal to $\delta$. 

Define a homotopy $\varphi_t : \overline{\mathcal{A}_{+\delta}} \to \R^4$ by 
\[
\varphi_t(a,b) = (R_{\beta(t,b)} \times R_{\beta(t,b)}) \varphi(R_{\beta(t,b)}(a), -b) , \qquad (a,b) \in \overline{\mathcal{A}_{+\delta}} ,
\]
for $t \in [0,1]$, where we note that the right side is well defined since $(R_{\beta(t,b)}(a), -b) \in \overline{\mathcal{A}_{-\delta}}$. The definition ensures $\varphi_0=\varphi$. Further, $\varphi_t$ is nonzero on $\partial \mathcal{A}_{+\delta}$ since $\varphi$ is nonzero on $\partial \mathcal{A}_{-\delta}$. Hence homotopy invariance of degree \cite[Theorem 2.3]{FG95} yields that 
\[
d(\varphi_0, \mathcal{A}_{+\delta}, 0) = d(\varphi_1, \mathcal{A}_{+\delta}, 0) .
\]
On the right side of this formula, observe that  
\begin{align*}
\varphi_1(a,b) 
& = (R_{(0,1)} \times R_{(0,1)}) \varphi(R_{(0,1)}(a), -b) \\
& = (R_{(0,1)} \times R_{(0,1)}) \varphi(a_1,-a_2,-b_1,-b_2) .
\end{align*}
By the multiplicative property of degree \cite[Theorem 2.10]{FG95}, the reflections $R_{(0,1)} \times R_{(0,1)}$ on the range each multiply the degree by $-1$, and so do the reflections $a_2 \mapsto -a_2$ and $b_1 \mapsto -b_1$ that preserve the domain $\mathcal{A}_{+\delta}$ and the reflection $b_2 \mapsto -b_2$ that maps $\mathcal{A}_{+\delta}$ to $\mathcal{A}_{-\delta}$. These five reflections together multiply the degree by $(-1)^5=-1$, and so 
\[
d(\varphi_1, \mathcal{A}_{+\delta}, 0) = - d(\varphi, \mathcal{A}_{-\delta}, 0) .
\]
Combining the equalities and recalling that $\varphi_0=\varphi$, we conclude as desired that 
\[
d(\varphi, \mathcal{A}_{+\delta}, 0) = -d(\varphi, \mathcal{A}_{-\delta}, 0) .
\]

\end{proof}

\section*{Acknowledgments}
Kim is supported by an RTG grant from the National Science Foundation ({\#}2135998). Laugesen's research is supported by grants from the Simons Foundation (\#964018) and the National Science Foundation ({\#}2246537). 

\appendix

\section{\bf The Robin problem on the unit disk} \label{sec:diskproblem}

The trial function $u_{w,C}$ constructed in \eqref{eq:trialfn} involves the second Robin eigenfunction $v$ of the unit disk, whose properties we now summarize. The disk eigenfunctions satisfy
\[
\begin{split}
\Delta v + \lambda v & = 0 \quad \text{in $\D$,} \\
\partial_\nu v + \alpha v & = 0 \quad \text{on $\partial \D$,} 
\end{split}
\]
where for simplicity we do not rescale the Robin parameter by $4\pi$ like in \autoref{th:main}. The parameter range $\alpha \in [-4\pi,4\pi]$ in that theorem corresponds to $\alpha \in [-1,1]$ in this appendix. For a plot of the Robin eigenvalues of the disk as a function of $\alpha$, we recommend \cite[Figure 4.2]{H17}.

The next proposition and \autoref{Robin_g_second} are taken from \cite[Section 5]{FreiLauS2020}. 
\begin{figure}
\includegraphics[scale=0.6]{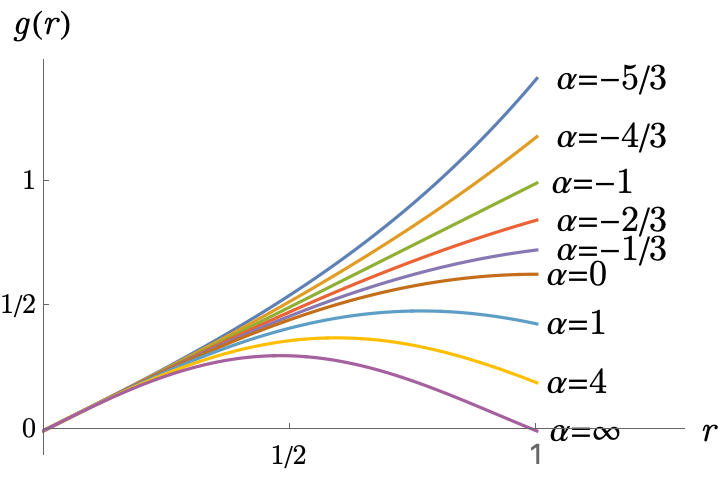}
\caption{\label{Robin_g_second}Plot of the radial part $g(r)$ of the second Robin eigenfunction $g(r) e^{i \theta}$ on the unit disk, for a range of $\alpha$-values. In terms of the $J_1$ Bessel function, $g(r)=(\text{const.})J_1(r\sqrt{\lambda_2(\D;\alpha)}\,)$ when $\alpha>-1$. \emph{Credit:} \cite[Figure 2]{FreiLauS2020}.}
\end{figure}
\begin{proposition}[Second Robin eigenfunctions of the disk]\label{basic2} Let $\alpha \in \R$. A complex-valued eigenfunction for $\lambda_2(\D;\alpha)$, can be taken in the form $v = g(r) e^{i \theta}$ where the radial part has $g(0)=0,g^\prime(0)>0,g(r)>0$ for $r \in (0,1)$, and $g(1)>0$. When $\alpha \leq 0$, $g$ is strictly increasing and $g^\prime>0$. When $\alpha > 0$, $g^\prime$ is initially positive and then negative. 

The eigenvalue $\lambda_2(\D;\alpha)$ is negative when $\alpha < -1$, equals zero at $\alpha=-1$, and is positive when $\alpha > -1$. 
\end{proposition}
The second eigenvalue of the disk has multiplicity $2$ since $e^{i\theta}$ yields both a sine and cosine mode. Thus the second and third eigenvalues agree. 

The second eigenvalue can be evaluated in terms of the Bessel function $J_1$, when $\alpha>-1$. We will not need the formula in this paper, but for the sake of completeness we recall (see \cite[Section 5]{FreiLauW2021} with dimension $n=2$) that $\lambda_2(\D;\alpha)=x(\alpha)^2$ where $x(\alpha) \in (0,j_{1,1})$ is the smallest positive solution of the Robin condition $x J_1^\prime(x)/J_1(x) = - \alpha$.

\bibliographystyle{plain}

\end{document}